\documentclass[12pt]{article}

\usepackage{amsthm}
\usepackage{amssymb}
\usepackage{amsmath}
\usepackage{comment}
\usepackage{thm-restate}
\usepackage{url} 
\usepackage{hyperref}
\usepackage[noabbrev,capitalise]{cleveref}

\usepackage{color}
\usepackage[normalem]{ulem}

\newcommand{\eps}{\varepsilon}
\newcommand{\mc}[1]{\mathcal{#1}}

\renewcommand{\v}{\textup{\textsf{v}}}
\newcommand{\e}{\textup{\textsf{e}}}
\renewcommand{\d}{\textup{\textsf{d}}}

\theoremstyle{plain}
\newtheorem{thm}{Theorem}[section]
\newtheorem{lem}[thm]{Lemma}

\newtheorem{cor}[thm]{Corollary}

\newtheorem{conj}[thm]{Conjecture}
\newtheorem{que}[thm]{Question}

{\noindent \emph{Proof.} {}{#1}{}}{\hfill
	$\Diamond$\vspace{1em}}

\theoremstyle{plain} 
\newcommand{\thistheoremname}{}
\newtheorem{genericthm}[section]{\thistheoremname}

\theoremstyle{definition}

\title{Connectivity and choosability of graphs with no $K_t$ minor}

\author{Sergey Norin\thanks{Department of Mathematics and Statistics, McGill University. Email: {\tt sergey.norin@mcgill.ca}. Supported by an NSERC Discovery grant.}
\and 
 Luke Postle\thanks{Department of Combinatorics and Optimization, University of Waterloo, Waterloo, Ontario, Canada. Email: {\tt lpostle@uwaterloo.ca}. Canada Research Chair in Graph Theory. Partially supported by NSERC under Discovery Grant No. 2019-04304, the Ontario Early Researcher Awards program and the Canada Research Chairs program.}
}

\begin{document}

\maketitle

\begin{center}
	\emph{Dedicated to the memory of Robin Thomas}
\end{center}

\begin{abstract}

In 1943, Hadwiger conjectured that every graph with no $K_t$ minor is $(t-1)$-colorable for every $t\ge 1$. While Hadwiger's conjecture does not hold for list-coloring, the linear weakening is conjectured to be true. In the 1980s, Kostochka and Thomason independently proved that every graph with no $K_t$ minor has average degree $O(t\sqrt{\log t})$ and thus is $O(t\sqrt{\log t})$-list-colorable. 

Recently, the authors and Song proved that every graph with no $K_t$ minor is $O(t(\log t)^{\beta})$-colorable for every $\beta > \frac 1 4$. Here, we build on that result to show that every graph with no $K_t$ minor is $O(t(\log t)^{\beta})$-list-colorable for every $\beta > \frac 1 4$.

Our main new tool is an  upper bound on the number of vertices in highly connected $K_t$-minor-free graphs: We prove that for every $\beta > \frac 1 4$, every $\Omega(t(\log t)^{\beta})$-connected graph with no $K_t$ minor has  $O(t (\log t)^{7/4})$ vertices.
\end{abstract}

\section{Introduction}

All graphs in this paper are finite and simple. Given graphs $H$ and $G$, we say that $G$ has \emph{an $H$ minor} if a graph isomorphic to $H$ can be obtained from a subgraph of $G$ by contracting edges. We denote the complete graph on $t$ vertices by $K_t$.

In 1943 Hadwiger made the following famous conjecture.

\begin{conj}[Hadwiger's conjecture~\cite{Had43}]\label{Hadwiger} For every integer $t \geq 0$, every graph with no $K_{t+1}$ minor is $t$-colorable. 
\end{conj}

Hadwiger's conjecture is widely considered among the most important problems in graph theory and has motivated numerous developments in graph coloring and graph minor theory. For an overview of major progress we refer the reader to~\cite{NPS19}, and to the recent survey by Seymour~\cite{Sey16Survey} for further background.

The following natural weakening of Hadwiger's conjecture has been considered by several researchers.

\begin{conj}[Linear Hadwiger's conjecture~\cite{ReeSey98,Kaw07, KawMoh06}]\label{c:LinHadwiger} There exists $C>0$ such that for every integer $t \geq 1$, every graph with no $K_{t}$ minor is $Ct$-colorable. 
\end{conj}

For many decades, the best general bound on the number of colors needed to properly color every  graph with no $K_t$ minor has been $O(t\sqrt{\log{t}})$, a result obtained independently by Kostochka~\cite{Kostochka82,Kostochka84} and Thomason~\cite{Thomason84} in the 1980s. The results of \cite{Kostochka82,Kostochka84,Thomason84}  bound the ``degeneracy" of graphs with no $K_t$ minor.
Recall that a graph $G$ is \emph{$d$-degenerate} if every non-null subgraph of $G$ contains a vertex of degree at most $d$. A standard inductive argument shows that every $d$-degenerate graph is $(d+1)$-colorable. Thus the following bound on the 
degeneracy of graphs with no $K_t$ minor gives a corresponding bound on their chromatic number and even their list chromatic number. 

\begin{thm}[\cite{Kostochka82,Kostochka84,Thomason84}]\label{t:KT} Every graph with no $K_t$ minor is $O(t\sqrt{\log{t}})$-degenerate.
\end{thm}

Very recently,  authors and Song~\cite{NPS19} improved the bound implied by \cref{t:KT} with the following theorem.

\begin{thm}[\cite{NPS19}]\label{t:ordinaryHadwiger}
For every $\beta > \frac 1 4$, every graph with no $K_t$ minor is $O(t (\log t)^{\beta})$-colorable.
\end{thm}

In \cite{NorSong19Odd} Song and the first author extended \cref{t:ordinaryHadwiger} to odd minors. 

In this paper we extend  \cref{t:ordinaryHadwiger} in a different direction -- to list coloring. Let  $\{L(v)\}_{v \in V(G)}$ be an assignment of lists of colors to  vertices of a graph $G$. We say that $G$ is \emph{$L$-list colorable} if there is a choice of colors $\{c(v)\}_{v \in V(G)}$ such that $c(v) \in L(v)$, and $c(v) \neq c(u)$ for every $uv \in E(G)$. 
A graph $G$ is said to be \emph{$k$-list colorable} if $G$ is $L$-list colorable for every list assignment $\{L(v)\}_{v \in V(G)}$ such that $|L(v)| \geq k$ for every $v \in V(G)$.  
Clearly every $k$-list colorable graph is $k$-colorable, but the converse does not hold. 

Voigt~\cite{Voigt93} has shown that there exist planar graphs which are not $4$-list colorable. Generalizing the result of~\cite{Voigt93}, Bar\'{a}t, Joret and Wood~\cite{BJW11} constructed graphs with no $K_{3t+2}$ minor which are not $4t$-list colorable for every $t \geq 1$. These results leave open the possibility that Linear Hadwiger's Conjecture holds for list coloring, as conjectured by Kawarabayashi and Mohar~\cite{KawMoh07}.

\begin{conj}[\cite{KawMoh07}]\label{c:ListHadwiger} There exists $C>0$ such that for every integer $t \geq 1$, every graph with no $K_{t}$ minor is $Ct$-list colorable. 
\end{conj}

\cref{t:KT} implies that every graph with no $K_t$ minor is $O(t\sqrt{\log{t}})$-list colorable, which until now was the best known upper bound for general $t$. 

Our main result extends \cref{t:ordinaryHadwiger} to list colorings.

\begin{restatable}{thm}{Main}\label{t:main} For every $\beta > \frac 1 4$,
every graph with no $K_t$ minor is $O(t (\log t)^{\beta})$-list-colorable.
\end{restatable}

In the course of proving Theorem~\ref{t:main}, we also prove a remarkably small upper bound on the number of vertices in $K_t$-minor-free graphs with connectivity $O(t(\log t)^\beta)$ for every $\beta > 1/4$ as follows.

\begin{restatable}{thm}{Connect}\label{t:connect} For every $\delta > 0$ and $1/2 \geq  \beta > 1/4$, there exists $C=C_{\ref{t:connect}}(\beta,\delta)>0$ such that if $G$ is $Ct (\log t)^{\beta}$-connected and has no $K_t$ minor then $\v(G) \leq t(\log t)^{3-5\beta + \delta}$.
\end{restatable}

Note that B\"{o}hme et al.\cite{BKMM09} proved a variant of \cref{t:connect} for graphs with connectivity linear in $t$. Namely, they show that for every $t$ there exists $N(t)$ such that every $\lceil \frac{31}{2}(t+1) \rceil$-connected graph $G$ with no $K_t$ minor satisfies $\v(G) \leq N(t)$. Their proof, however, relies on the Robertson-Seymour graph minor structure theorem and does not provide a reasonable bound for $N(t)$.    

\subsubsection*{Outline of Paper}

The proof of \cref{t:connect} reuses the main tools used to establish \cref{t:ordinaryHadwiger} in \cite{NPS19}.  \cref{t:newforced} below shows that any dense enough graph with no $K_t$ minor contains a reasonably small subgraph with essentially the same density. Meanwhile, \cref{t:minorfrompieces} guarantees that any graph with appropriately high connectivity containing many such dense subgraphs has a $K_t$ minor. These and other necessary tools are introduced in \cref{s:prelim}.

To play these two results against each other, we need a new ingredient: an extension of \cref{t:KT} to upper bound the density of asymmetric bipartite graphs with no $K_t$ minor. We prove such a bound in \cref{s:density}. In \cref{s:connect}  we use this bound to derive \cref{t:connect}.

In \cref{s:lower} we use random constructions to show that the bounds in \cref{s:density} are tight up to the constant factor and establish lower bounds on the maximum size of a graph with no $K_t$ minor and given connectivity.

In \cref{s:alon} we generalize a bound of Alon~\cite{Alon92} on choosability of complete multipartite graphs to prove a bound on choosability of a graph in terms of its number of vertices and Hall ratio. In \cref{s:final} we use this bound and \cref{t:connect} to establish \cref{t:main}. \cref{s:remarks} contains concluding remarks.

\subsubsection*{Notation}

We use largely standard graph-theoretical notation. We denote by $\v(G)$ and $\e(G)$ the number of  vertices and edges of a graph $G$, respectively, and denote by $\d(G)=\e(G)/\v(G)$ the \emph{density} of a non-null graph $G$. We use $\chi_{\ell}(G)$ to denote the list chromatic number of $G$, and $\kappa(G)$ to denote the (vertex) connectivity of $G$. We write $H \prec G$ if $G$ has an $H$ minor.   We denote by $G[X]$ the subgraph of $G$ induced by a set $X \subseteq V(G)$. For disjoint subsetes $A,B\subseteq V(G)$, we let $G(A,B)$ denote the bipartite subgraph induced by $G$ on the parts $(A,B)$. For $F \subseteq E(G)$ we denote by $G/F$ the minor of $G$ obtained by contracting the edges of $F$.

For a positive integer $n$, let $[n]$ denote the set $\{1,2,\ldots,n\}$. The logarithms in the paper are natural unless specified otherwise.

We say that vertex-disjoint subgraphs $H$ and $H'$ of a graph $G$ are \emph{adjacent} if there exists an edge of $G$ with one end in $V(H)$ and the other in $V(H')$, and $H$ and $H'$ are \emph{non-adjacent}, otherwise.

A collection $\mc{X} = \{X_1,X_2,\ldots,X_h\}$ of pairwise disjoint subsets of $V(G)$ is a \emph{model of a graph $H$ in a graph $G$} if $G[X_i]$ is connected for every $i \in [h]$, and there exists a bijection $\phi: V(H) \to [h]$, such that $G[X_{\phi(u)}]$ and $G[X_{\phi(v)}]$ are adjacent for every $uv \in E(H)$. It is well-known and not hard to see that $G$ has an $H$ minor if and only if there exists a model of $H$ in $G$.

\section{Preliminaries and Previous Results}\label{s:prelim}

For this paper, we will need two classical results on $K_t$ minor-free graphs: the first, a lower bound on their independence number; the second, an upper bound on their density.

\begin{thm}[\cite{DucMey82}]\label{t:DucMey}
	Every graph $G$ with no $K_t$ minor has an independent set of size at least $\frac{\v(G)}{2(t-1)}$. 
\end{thm}

\begin{thm}[\cite{Kostochka82}]\label{t:density}
 	Let $t \geq 2$ be an integer. Then every graph $G$  with $\d(G) \geq 3.2 t \sqrt{\log t}$ has a $K_t$ minor.
\end{thm}

We also need the following results from Norin, Postle and Song~\cite{NPS19}.

\begin{thm}[\cite{NPS19}]\label{t:newforced} For every $\delta > 0$ there exists $C=C_{\ref{t:newforced}}(\delta) > 0$ such that for every $D > 0$ the following holds. Let $G$ be a graph with $\d(G) \ge C$, and let $s=D/\d(G)$.  Then $G$ contains at least one of the following: \begin{description}
		\item[(i)] a minor $J$ with $\d(J) \geq D$, or
		\item[(ii)] a subgraph $H$ with $\v(H) \leq s^{1+\delta} CD$ and $\d(H) \geq s^{-\delta}\d(G)/C$.
	\end{description}  
\end{thm}


\begin{thm}[\cite{NPS19}]\label{t:minorfrompieces} For every $\beta\in [\frac 1 4, \frac 1 2]$, there exists $C=C_{\ref{t:minorfrompieces}} >1$ satisfying the following. 
	Let $G$ be a graph with $\kappa(G) \geq  Ct(\log t)^{\beta}$, and let $r \geq (\log t)^{1-2\beta}/2$ be an integer. If 
	there exist pairwise vertex disjoint subgraphs $H_1,H_2,\ldots,H_{r}$ of $G$ such that $\d(H_i) \geq Ct(\log t)^{\beta}$ for every $i \in [r]$  then $G$ has a $K_t$ minor. 
\end{thm}

Note that Theorem~\ref{t:minorfrompieces} was stated only for $\beta = \frac 1 4$ in~\cite{NPS19}, however the same proof works for every $\beta \in [\frac 1 4, \frac 1 2]$.

\section{Asymmetric density}\label{s:density}

In this section we use variants of arguments of Thomason~\cite{Thomason84,Thomason01} to establish an upper bound on the density of assymetric bipartite graphs with no $K_t$ minor.

\begin{lem}\label{l:dense}
Let $t$ be a positive integer, let $G$ be a graph with $n =\v(G) \geq 9t$, let $q= 1- \frac{\e(G)}{\binom{n}{2}}$ and $l = \left\lfloor \frac{n}{9t} \right\rfloor$. If \begin{equation}\label{e:dense} 6t(70q)^{l^2} \leq 1,\end{equation} then $G$ has a $K_t$ minor.
\end{lem} 

\begin{proof} By definiton of $q$, $G$ contains $\binom{n}{2}q$ non-edges. Thus  there exists a set $Z$ of $\lfloor n/3 \rfloor $ vertices of $G$ such that each vertex in $Z$ has at most $2qn$ non-neighbors in $G$. 
	
Given $v \in Z$,  consider $X \subseteq Z - \{v\}$ with $|X|=l$ chosen uniformly at random. 
Then  the probability that $v$ has no neighbor in $X$ is at most $(2qn/(|Z|-1))^l \leq (7q)^l$. 

It follows that if $X \subseteq Z$ with $|X|=l$ is chosen uniformly at random, then the expected number of vertices in $Z-X$ with no neighbor in $X$ is at most $n(7q)^l$. We say that a set $X$ is \emph{good} if at most $3n(7q)^l$ vertices in $Z-X$ have no neighbor in $X$. By Markov's inequality the probability that the set $X$ as above  is good is at least $2/3$.
	
	Given a good set $X \subseteq Z$, suppose that a set $Y$ of size $l$ is selected from $Z-X$ uniformly at random.  Then the probability that no vertex of $Y$ is adjacent to a vertex of $X$ is at most $$\left( \frac{3n(7q)^{l}}{|Z|-l} \right)^l \leq  (70q)^{l^2}.$$ 
	
	We now select disjoint subsets  $X_1,X_2,\ldots,X_{2t},Y_1,Y_2,\ldots,Y_{t}$ of $Z$ such that $|X_i|,|Y_j|=l$ uniformly at random. We say that a pair $(i,j) \in [2t] \times [t]$ is \emph{unfulfilled} if there does not exist $\{u,v\} \in E(G)$ with $u \in X_i$, $v \in Y_j$. We say that $X_i$ is \emph{perfect} if  $(i,j)$ is not unfulfilled for every $j \in [t]$.
	
	By the calculations above, if $X_i$ is good then the expected number of unfulfilled pairs $(i,j)$  is at most  $10^l(7q)^{l^2}t \leq 1/6$ by (\ref{e:dense}). Therefore the probability that $X_i$ is perfect is at least $1/2$. Thus there exists a choice of sets $\{X_i\}_{i \in [2t]}, \{Y_j\}_{j \in [t]}$ as above, such that at least $t$ of the sets $X_1,X_2,\ldots,X_{2t}$ are perfect. Thus we may assume that $X_1,\ldots,X_t$ are perfect.
	
	Note that 
	every two non-adjacent vertices in $Z$ have at least $(1-4q)n \geq 2/3n$ common neighbors. In particular, every two such vertices have more than $|Z|$ common neighbors in $V(G)-Z$. Thus we can greedily construct pairwise disjoint $B_1,B_2,\ldots,B_t \subseteq V(G)$ such that $X_i \cup Y_i \subseteq B_i$, and $G[B_i]$ is connected for every $i \in [t]$ . These sets form a model of $K_t$ as desired.
\end{proof}

\begin{thm}\label{t:logbip}
	There exists $C=C_{\ref{t:logbip}}>0$ such that for every $t \geq 3$ and every bipartite graph $G$ with bipartition $(A,B)$ and no $K_t$ minor we have
	\begin{equation} \label{e:logbip}
	\e(G) \le C t\sqrt{\log t} \sqrt{|A||B|}  + (t-2)\v(G).
	\end{equation}	 
\end{thm}	 

\begin{proof} We show that $C = 6400 > 4(20)^2(1+\log(20))$ satisfies the lemma. Suppose for a contradiction that there exists a bipartite graph $G$ with bipartition $(A,B)$ with no $K_t$ minor such that (\ref{e:logbip}) does not hold. Choose such $G$ with $\v(G)$ minimum. Let $\alpha=\sqrt{|A|/|B|}$ and consider $v \in A$. By the choice of $G$, we have $$ \e(G \setminus v) \le C t\sqrt{\log t} \sqrt{(|A|- 1)|B|}  + (t-2)(\v(G) - 1), $$ and so
	\begin{align*} \deg(v) &= \e(G) - \e(G \setminus v) \\ &\geq C t\sqrt{\log t} (\sqrt{|A||B|} - \sqrt{(|A|-1)|B|}) + t-2 \\ &\geq \frac{C}{2} \alpha^{-1}t\sqrt{\log t} + t-2. \end{align*}
	Similarly, 	$$\deg(v)\geq \frac{C}{2}  \alpha t\sqrt{\log t} + t-2$$ for every $v \in B$.
	 Assume $|A| \geq |B|$, without loss of generality. Then there exists $v_0 \in A$ such that $\deg(v_0) \leq 7 t \sqrt{\log{t}} \leq \frac{C}{4}  \alpha t\sqrt{\log t},$ as otherwise $G$  has a $K_t$ minor by Theorem~\ref{t:density}. 
	
	Fix an arbitrary pair of neighbors $u_1,u_2 \in B$ of $v_0$ and consider the graph $G'$ obtained from $G$ by deleting $v_0$ and identifying $u_1$ and $u_2$. As $G'$ is a minor of $G$, we have that $G'$ has no $K_t$ minor, and so $$\e(G') \leq C t\sqrt{\log t} \sqrt{(|A|- 1)(|B|-1)}  + (t-2)(\v(G) - 2), $$ by the choice of $G$. Let $d(u_1,u_2)$ denote the number of common neighbors of $u_1$ and $u_2$ in $A - \{v_0\}$. As $\e(G)-\e(G') = \deg(v_0) + d(u_1,u_2)$ the bounds on $\e(G),\e(G')$ and $\deg(v_0)$ above imply  that $$ d(u_1,u_2) \geq  \frac{C}{4} \alpha  t\sqrt{\log t} =: s.$$
	
	Let $n = \lceil  \alpha^{-1}t\sqrt{\log t} + t-2 \rceil \geq t-1$, and let $X$ be a set of  $n$ arbitrary neighbors of $v_0$. For every $v \in A - v_0$ such that $v$ has a neighbor in $X$, we choose such a neighbor $u$ uniformly independently at random, and contract $v$ onto $u$. Let $H$ be the random graph induced on $X$ obtained via this procedure. The probability that any two given vertices in $X$ are non-adjacent in $H$ is at most
	$$q:=\left(1-\frac{2}{n}\right)^{s} \le e^{-2s/n}.$$
	If $\binom{n}{2}q < 1$, then with positive probability $H$ is complete, and so $G$ contains a complete minor on $n+1$ vertices, a contradiction. Thus  we assume that $\binom{n}{2}q \geq 1$, implying $2\log n \geq 2s/n$. Moreover,  $sn \geq Ct^2\log{t}/4$ by definition of $s$ and $n$. It follows that $n^2 \log{n} \geq Ct^2\log{t}/4$ implying $n \geq 20t$.\footnote{Otherwise, $n^2 \log{n} \leq (20)^2 t^2(\log t + \log 20) \leq (20)^2(\log(20)+1) t^2\log{t} < Ct^2\log{t}/4$.} As $\alpha \leq 1$ from definition of $n$ we have $n \leq 2t\sqrt{\log t}$. Combining these inequalities we have and so $2s/n \geq C/8$ and $q < 1/(70)^2$. 
	
	Let $l = \lfloor \frac{n}{9t} \rfloor$. Then \begin{equation}
	\label{e:l}
	l \geq \frac{n}{18t} \geq \frac{1}{18}\alpha^{-1}\sqrt{\log t}
	\end{equation}
	and 
	\begin{align*}
	(70q)^{l^2} &\leq q^{l^2/2} \leq \exp\left( -\frac{sl^2}{n}\right)
	\leq \exp\left(-\frac{sl}{18t}\right) \\ &\leq \exp\left(-\frac{C}{72} \log t\right)\leq \frac{1}{t^3} \leq \frac{1}{6t}.
	\end{align*}
	Thus (\ref{e:dense}) holds for $H$, and thus $H$ contains a $K_t$ minor by \cref{l:dense}, a contradiction.
\end{proof}	

Note that the graph $K_{a,t-2}$ has no $K_t$ minor for any integer $a$ showing that the term $(t-2)\v(G)$ in (\ref{e:logbip}) is necessary. In \cref{s:lower} we show that the bound in \cref{t:logbip} is tight for all values of $|A|,|B|$ up to the constant factor.

\section{Proof of \cref{t:connect}}\label{s:connect}

In this section, we prove Theorem~\ref{t:connect}, which we restate for convenience.

\Connect*
\begin{proof}[Proof of Theorem~\ref{t:connect}] We assume without loss of generality that $\delta < 1/4$.  It suffices to show that there exist  $C, t_0=t_0(\delta)$,  such that for all positive integers $t \geq t_0$,  every graph $G$ with $\kappa(G) \geq Ct(\log t)^{\beta}$ and no $K_t$ minor satisfies $$\v(G) \le t(\log t)^{3-5\beta+\delta}.$$

Let $\delta' = \delta/3$.	Let $C_1 = C_{\ref{t:newforced}}(\delta')$, and let  $C=\max\{ 4 C_{\ref{t:logbip}}, C_{\ref{t:minorfrompieces}}\}$. We choose $t_0 \gg C,C_1,1/\delta$ implicitly to satisfy the inequalities appearing throughout the proof. 

Let $k = Ct(\log t)^{\beta}$ and let $G$ be a graph with $\kappa(G) \geq k$ and no $K_t$ minor.
Choose a maximal collection $H_1,H_2,\ldots,H_{r}$ of pairwise vertex disjoint subgraphs of $G$ such that $\d(H_i) \geq Ct(\log t)^{\beta-\delta'}$ and $\v(H_i) \leq t(\log t)^{1-\beta+\delta'}$. 	Since $G$ has no $K_t$ minor, it follows from \cref{t:minorfrompieces} that $r < (\log t)^{1-2\beta+2\delta'}/2$. Let $X = \cup_{i \in [r]}V(H_i)$. Then $|X| < t (\log t)^{2-3\beta+3\delta'} = t(\log t)^{2-3\beta + \delta}$.
	
Let $G' = G\setminus X$. First suppose that $\d(G')\ge k/4$. Let $D= 3.2 t \sqrt{\log t}$. We apply \cref{t:newforced} to $\delta'$, $D$ and $G'$. If $G'$ has a minor $J$ with $\d(J) \geq D$ then $G'$ has a $K_{t}$ minor by \cref{t:density}, contradicting the choice of $G$.	Thus there exists a subgraph $H$ of $G'$ such that  $\v(H) \leq s^{1+\delta'} C_1D$ and $\d(H) \geq s^{-\delta'}\d(G')/C_1$, where $s=D/\d(G') \leq 13(\log{t})^{1/2-\beta}$.  It is easy to check that for large enough $t$ the above conditions imply $\d(H) \geq  Ct(\log t)^{\beta-\delta'}$ and $\v(H) \leq t(\log t)^{1-\beta+\delta'}$. Thus the collection  $\{H_1,H_2,\ldots,H_{r},H\}$ contradicts the maximality of $\{H_1,H_2,\ldots,H_{r}\}$.
	
	So we may assume that $\d(G') < k/4$. That is, $\e(G') < (k/4) \v(G')$. Since $\kappa(G) \geq k$, every vertex in $V(G')$ has degree at least $k$ in $G$. It follows that
	\begin{equation}
	\label{e:connect1}
	e(G(X, V(G'))) \ge \frac{k}{2} \v(G').
	\end{equation} Yet since $G$ has no $K_t$ minor, we have by \cref{t:logbip} applied to $G(X,V(G'))$ that
	\begin{equation}
	\label{e:connect2}
	e(G(X, V(G'))) \le  C_{\ref{t:logbip}} t\sqrt{\log t}\sqrt{|X|\v(G')} + t(|X|+\v(G')).
	\end{equation}
	If $\v(G') \leq |X|$ then $$\v(G) \leq 2t(\log t)^{2-3\beta + \delta} \leq t(\log t)^{3-5\beta+\delta}$$ for sufficiently large $t$, as desired.
	Thus we assume $\v(G') \geq |X|$. Combining (\ref{e:connect1}) and (\ref{e:connect2}) we have 	\begin{equation}\label{e:connect3} (k/2-2t)\v(G') \leq  C_{\ref{t:logbip}}t\sqrt{\log t}\sqrt{|X|\v(G')}.	\end{equation} 
	Assuming that $t$ is large enough, we have that $k \geq 8t$, and so $k/2-2t \geq k/4$. Thus the above implies 
	$$\v(G') \leq (4 C_{\ref{t:logbip}})^2t^2\log t \cdot\frac{|X|}{k^2} \leq \frac{|X|\log{t}}{(\log{t})^{2\beta}} \leq  t(\log t)^{3-5\beta + \delta},$$ 
	as desired.
\end{proof}

\section{Lower bounds}\label{s:lower}

In this section we prove lower bounds on the density of asymmetric bipartite graphs with no $K_t$ minor and on the size of such graphs with given connectivity.

For $0 \leq p \leq 1$ and  pair of integers $a,b >0$ we denote by ${\bf G}(a,b,p)$ a random bipartite graph with bipartition $(A,B)$ where $A$ and $B$ are disjoint sets with $|A|=a$, $|B|=b$ and the edges between $A$ and $B$ are chosen independently at random with probability $p$. The next lemma mirrors a computation first used by Bollobas, Caitlin and Erd\H{o}s~\cite{BCE80} to compute the size of the largest minor in a random graph. 

\begin{lem}\label{l:random}
	For every $\eps >0$ there exists  $t_0$, such that for all $0 < p < 1$ and integers $t \geq t_0, a,b \geq 0$ such that $ab \leq (1-\eps)\frac{1}{-2\log (1-p)} t^2\log t$, we have
	$$\Pr [K_{t} \mathrm{\;is \;a \:minor \:of\;} {\bf G}(a,b,p)  ] \leq e^{-t^{\eps}/3}.$$
\end{lem}

\begin{proof}
	Let $(A,B)$ be the bipartition of $G={\bf G}(a,b)$ as in the definition, and let $(A_1,\ldots,A_t)$ and $(B_1, \ldots,B_t)$ be  partitions of $A$ and $B$, respectively. Let $a_i = |A_i|, b_i = |B_i|$ for $i \in [t]$. Let $q = 1-p$. Then the probability that $G$ does not contain an edge from $A_i \cup B_i$ to $A_j \cup B_j$ is $q^{a_ib_j+a_jb_i}$. Thus we can upper bound the probability that $\{A_i \cup B_i\}_{i \in [t]}$ is a model of $K_t$ in $G$ by
	\begin{align*}
	\prod_{\{i,j\} \subseteq [t]}&\left(1 - q^{a_ib_j+a_jb_i} \right) \leq \exp \left(-\sum_{ \{i,j\} \subseteq [t]}q^{a_ib_j+a_jb_i} \right) \\ 
	& \leq \exp\left( -\binom{t}{2} q^{(\sum_{{\{i,j\} \subseteq [t]}}(a_ib_j+a_jb_i))/\binom{t}{2}}\right)  \\ &\leq \exp\left( -\binom{t}{2} q^{ab/\binom{t}{2}}\right)\leq \exp\left( -\binom{t}{2} q^{-(1-\eps)\log t/\log q} \right) \\ &= \exp\left(-\frac{(t-1)t^{\eps}}{2}\right).
	\end{align*}
	Suppose $a \geq b$ without loss of generality. If $b \leq t-2$ then $G$ has no $K_t$ minor. Thus we may assume $b \geq t-1 $, implying $a+b \leq t\log t$ for large enough $t$. The number of partitions $(A_1,\ldots,A_t)$ and $(B_1, \ldots,B_t)$ as above can then be loosely upper bounded by $t^{a+b} \leq \exp(t\log^2{t})$.
	By the union bound we deduce that the probability that $G$ has a $K_t$ minor is at most
	$$\exp\left(t\log^2{t}-\frac{(t-1)t^{\eps}}{2}\right) \leq \exp\left(-t^{\eps}/3\right),$$
	as desired, where the last inequality holds for $t$ large enough.
\end{proof}
Let
\begin{equation*}
\lambda := \max_{x >0} \frac{1-e^{-x}}{\sqrt{x}}=0.63817\ldots.
\end{equation*}
be the constant which appears, in particular, in the optimal bound on the asymptotic density of graphs with no $K_t$ minor established by Thomason~\cite{Thomason01}.

\begin{cor}\label{c:random}
	For every $\eps >0$ there exists  $C$, such that for all integers $a,b \geq t \geq C$ such that \begin{equation}\label{e:ab}
ab \geq C t^2\log t
	\end{equation} there exists a bipartite graph $G$ with bipartition $(A,B)$ such that $|A|=a,|B|=b$, $G$ has no $K_t$ minor and  
	$$\e(G) \geq (1-\eps) \frac{\lambda}{\sqrt{2}} t\sqrt{\log t}\sqrt{|A||B|   }.$$
\end{cor}

\begin{proof} Let $C$ be chosen implicitly to satisfy the inequalities throughout the proof, and let $\eps' =\eps/4$.  
	
	Assume $a \geq b$, without loss of generality. Note that $K_{a,t-2}$ has no $K_t$ minor, and hence the corollary holds if 
	$$(1-\eps) \frac{\lambda}{\sqrt{2}} t\sqrt{\log t}\sqrt{ab} \leq (t-2)a .$$ 
	Thus we may assume \begin{equation}\label{e:b}
	b \log t  \geq a
	\end{equation} given $C$ is large enough.
	
	Let $x$ be such that $\lambda = \frac{1-e^{-x}}{\sqrt{x}}$, and let $p = 1-e^{-x}$. Let $$k = \left\lceil \sqrt{(1-\eps') \frac{-2\log (1-p)ab}{t^2\log t} }\right\rceil.$$
	Let $a'= \lfloor a/k \rfloor,b'= \lfloor b/k \rfloor$. By (\ref{e:ab}) and (\ref{e:b}), we have $b' \geq 1/\eps'$ given that $C$ is large enough. In particular, this implies that $b' \geq (1-\eps')b/k$ and $a' \geq (1-\eps')a/k$.
	
	By the Chernoff bound
	$$\Pr [ \e ({\bf G}(a',b',p)) \leq (1-\eps')pa'b ' ] \leq e^{- (\eps')^2 pa'b'/2} \leq  e^{-p/2}.$$
	Combining this observation \cref{l:random} we deduce that for large enough $C$, there exists a bipartite graph $G'$ with no $K_t$ minor and a bipartition $(A',B')$ such that
	$|A'|=a',|B'|=b'$ and $\e(G') \geq (1-\eps')pa'b '$. 
	
	We obtain $G$ by taking $k$ vertex disjoint copies of $G'$ (and adding isolated vertices if necessary).  Then 
	\begin{align*}
	\e(G) &\geq (1-\eps')kpa'b ' \geq (1-\eps')^3p \frac{ab}{k} \\ &\geq (1-\eps')^4 p \frac{ab}{\sqrt{\frac{-2\log (1-p)ab}{t^2\log t}}} \\ & = (1 - \eps')^4\frac{1}{\sqrt 2}\frac{1-e^{-x}}{\sqrt{x}}t\sqrt{\log t}\sqrt{ab} \\ &\geq 
	 (1-\eps) \frac{\lambda}{\sqrt{2}} t\sqrt{\log t}\sqrt{|A||B|   },
	\end{align*}
	as desired.
\end{proof}

\cref{c:random} shows that the bound in \cref{t:logbip} is tight up to the constant factor. We believe that the constant in \cref{c:random} is likely asymptotically optimal.

Next we establish a lower bound on the size of graphs with given connectivity and no $K_t$ minor.
A standard easy argument shows that  with high probability  ${\bf G}(a,b,1/2)$ is $(1 - o(1))(b/2)$-connected for $a \geq b$, as long as $a$ is not too large compared to $b$, as formalised in the next lemma.

\begin{lem}\label{l:random2}
For every $0 < \eps < 1$ and all integers $a \geq b \geq 1$ such that $a(a+1) \leq \exp(\eps^2b/32)$ we have
$$\Pr \left[\kappa({\bf G}(a,b,1/2)) < (1 - \eps)\frac{b}{2}\right] \leq \exp(-\eps^2b/64)$$
\end{lem}
\begin{proof}
Again let $(A,B)$ be the bipartition of $G={\bf G}(a,b)$ as in the definition, and  let $k = (1-\eps)b/2$.
By the Chernoff bound
$$\Pr \left [\deg(v) < k \right]  \leq \exp(-\eps^2b/8),$$
for every $v \in A$, and the probability that a pair of vertices $v_1,v_2 \in A$ share at most  $k/2$ neighbors is at most $\exp(-\eps^2b/32).$ Analogous bounds with $b$ replaced by $a$ hold for vertices in $B$. Thus with probability at least $$1 - a(a+1)\exp(-\eps^2b/32) \geq 1 - \exp(-\eps^2b/64) $$ 
 every vertex of $G$ has degree at least $k$ and every pair of vertices of $G$ on the same side of the bipartion share more than $k/2$ neighbors. 
 
 These properties are sufficient to guarantee that $\kappa(G) \geq k$ implying the lemma. Indeed, consider $X \subseteq V(G)$ with $|X| < k$ and assume first $|A \cap X|< k/2$. Then every pair of vertices of $B - X$ share a neighbor in $A - X$, and so $B-X$ lies in a single component of $G \setminus X$. As every vertex in  $A-X$ has a neighbor in $B-X$ it follows that $G \setminus X$ is connected. The case $|B \cap X|< k/2$ is completely analogous.
\end{proof}

\begin{cor}\label{c:random2}
	There exist $\eps, t_0 >0$ such that for all integers $t \geq t_0$ and every integer $k \le \eps t \sqrt{\log t}$ there exists a graph $G$ with $\kappa(G) \geq k$ and $$\v(G) \geq \eps\frac{t^2\log t}{k}.$$
\end{cor}

\begin{proof}
If $k \leq t-2$ then $K_{a,t-2}$ satisfies the corollary for $a$ large enough. Otherwise, let $b = 3k$ and let $a = \lceil \frac{1}{6} \frac{t^2\log t}{k} \rceil$. Then by Lemmas ~\ref{l:random} and \ref{l:random2} $G = {\bf G}(a,b,1/2)$ has no $K_t$ minor and satisfies $\kappa(G) \geq k$ for large enough $t$ and small enough $\eps$. As $\v(G) \geq a \geq \frac{1}{6} \frac{t^2\log t}{k}$ the corollary follows for $\eps \leq 1/6$.
\end{proof}

It follows from \cref{c:random2} that the bound $t(\log t)^{3-5\beta +o(1)}$ in \cref{t:connect} can not be improved beyond $O(t(\log t)^{1-\beta})$.

\section{List coloring vs. Hall ratio}\label{s:alon}

Let $K_{m*r}$ denote the complete $r$-partite graph with $m$ vertices in every part. Alon~\cite{Alon92} has proved the following.

\begin{thm}\label{t:alon}
	There exists $C_{\ref{t:alon}} > 0$ such that for every $m \geq 2$ $$\chi_l(K_{m*r}) \leq C_{\ref{t:alon}}r \log(m).$$
\end{thm} 

The \emph{Hall ratio} of a graph $G$ is defined to be $\max_{H\subseteq G} \left\lceil \frac{\v(H)}{\alpha(H)} \right\rceil$. We use Theorem~\ref{t:alon} to prove the following theorem relating the list chromatic number of a graph with its Hall ratio and number of vertices.

\begin{thm}\label{t:listHall} There exists $C =C_{\ref{t:listHall}}  > 0$ satisfying the following.
	Let $\rho \geq 3$, and let $G$ be a graph with the Hall ratio at most $\rho$, and let $n = \v(G)$. If $n \geq 2\rho$, then
	$$\chi_l(G) \leq C\rho\log^2\left(\frac{n}{\rho}\right).$$ 	
\end{thm}

\begin{proof}  We show by induction on $n$ that $C=\max\{16 C_{\ref{t:alon}}, \frac{3e}{\log 2}\}$ satisfies the theorem. The theorem clearly holds for $n \leq 3e\rho$ for this choice of $C$, so we assume that $n \geq 3e\rho$ for the induction step. 
	
	Consider an assignment of lists $\{L(v)\}_{v \in V(G)}$ of colors of size $l \geq C\rho\log^2(\frac{n}{\rho})$ to the vertices of $G$.
	Select a subset $L_1$ of colors by choosing every color independently at random with probability $1/\log\left(\frac{n}{\rho}\right)$. Let $L_1(v)$ denote the set of colors in $L_1$ assigned to $v$.
	By the Chernoff bound, we have
	$$ \Pr\left [ |L_1(v)| \leq \frac{C}{2}\rho \log\left(\frac{n}{\rho}\right)\right] \leq \exp\left(-\frac{1}{8}C\rho\log\left(\frac{n}{\rho}\right)\right) < \frac{1}{2n},\footnote{The last inequality holds as $C \geq 8$ and $$\rho \log\left(\frac{n}{\rho}\right) \geq \rho +   \log\left(\frac{n}{\rho}\right) + 1 \geq \log{2}+\log{\rho} +  \log\left(\frac{n}{\rho}\right)  =  \log(2n)$$} $$ 
	and, similarly, $$ \Pr \left[ |L_1(v)| \geq \frac{3}{2} C\rho\log\left(\frac{n}{\rho}\right) \right]  < \frac{1}{2n}.$$
	Thus by the union bound, with positive probability none of these events happen for any vertex $v$ of $G$. So we may assume that for every vertex $v$ of $G$, we have $$ \frac{C}{2}\rho\log\left(\frac{n}{\rho}\right)  \leq |L_1(v)|   \leq \frac{3}{2}C\rho\log\left(\frac{n}{\rho}\right).$$
	
	Let $s = \left\lfloor \frac{n}{e\rho}\right\rfloor$.
	We repeatedly select disjoint independent sets $X_1,X_2,\ldots,X_k$ in $G$ of size $s$, where $k = \left\lceil \left(1 -\frac{1}{e}\right)\frac{n}{s} \right\rceil$. This is possible, as $G - \cup_{j=1}^{i}X_j$ is a subgraph of $G$ on at least $n-(k-1)s \geq n/e$ vertices for every $i \in [k-1]$, and so $G - \cup_{j=1}^{i}X_j$ contains an independent set of size at least $s$ by definition of $\rho$.  
	
	Note that $s \geq 3$ by the choice of $n$, and so $s \geq \frac{3n}{4e \rho}$. Thus $$k \leq \left\lceil \frac{4(e-1)}{3} \rho \right\rceil \leq 4 \rho.$$ 
	Let $X = \cup_{i=1}^{k} X_i$. By Theorem~\ref{t:alon} and the above bounds on $k$ and $s$, we have $$\chi_l(G[X]) \leq C_{\ref{t:alon}}k\log{s} \leq 4C_{\ref{t:alon}}\rho \left(\log\left(\frac{n}{\rho}\right) + 2 \right) \leq 8C_{\ref{t:alon}} \rho \log\left(\frac{n}{\rho}\right) $$
	and so there exists an $L_1$-coloring $\phi_1$ of $G[X]$,as $C \geq 16 C_{\ref{t:alon}}$.
	
	By the induction hypothesis, we have
	\begin{align*}
	\chi_l(G \setminus X) &\leq C\rho\log^2\left(\frac{n}{e\rho}\right) = C \rho \left(\log\left(\frac{n}{\rho}\right) - 1 \right)^2\\
	&= C\rho\log^2\left(\frac{n}{\rho}\right) - 2 C \rho \log\left(\frac{n}{\rho}\right) +  C \rho\\
	&\leq  C\rho\log^2\left(\frac{n}{\rho}\right) -   \frac{3}2{}C\rho\log\left(\frac{n}{\rho}\right),
	\end{align*}
	where the last inequality follows since $n\ge e^2\rho$. Thus $G\setminus X$ has an $L_2$-coloring $\phi_2$, where $L_2(v) = L(v)\setminus L_1(v)$ for every $v\in V(G\setminus X)$. But then $\phi_1\cup \phi_2$ is an $L$-coloring of $G$ as desired.
\end{proof}	

\begin{cor}\label{c:listsmall}
There exists $C= C_{\ref{c:listsmall}} > 0$ satisfying the following. If $G$ is a graph with no $K_t$ minor for some $t \ge 2$ and $\v(G) \geq 4t$, then
	$$\chi_l(G) \leq Ct\log^2\left(\frac{\v(G)}{2t}\right).$$ 
\end{cor}
\begin{proof}
It follows from Theorem~\ref{t:DucMey} that the Hall ratio of $G$ is at most $2t$. The corollary now follows from Theorem~\ref{t:listHall} with $\rho = 2t$. 
\end{proof}

\section{Proof of \cref{t:main}}\label{s:final}


Before we prove Theorem~\ref{t:main}, we first need the following definition and lemma. If $G$ is a graph and $X\subseteq V(G)$, then the \emph{coboundary} of $X$ in $G$ is $(\bigcup_{v\in X} N(v))\setminus X$.

\begin{lem}\label{l:contract}
Let $k\ge 1$. If $G$ is a non-empty graph with minimum degree $d\ge 6k$, then there exists a non-empty $X\subseteq V(G)$ and a matching $M$ from the coboundary $Y$ of $X$ to $X$ that saturates $Y$ such that $|Y|\le 3k$ and $G[X\cup Y]/M$ is $k$-connected. 
\end{lem}

\begin{proof}
Suppose not. Let $X$ be a non-empty subset of $V(G)$ such that the coboundary $Y$ of $X$ has size at most $3k$ and subject to that $|X|$ is minimized. Such an $X$ exists as $V(G)$ has empty coboundary. 

We claim that there exists a matching $M$ from $Y$ to $X$ that saturates $Y$. Suppose not. By Hall's theorem, there exists $S\subseteq Y$, such that $|N(S)\cap X| < |S|$. If $N(S)\cap X = X$, then $|X| < 3k$ and hence every vertex in $X$ has degree at most $|X|+|Y|-1 < 6k \leq d$, a contradiction. So we may assume that $X' = X \setminus (N(S)\cap X)\ne \emptyset$. But then $X'$ has a coboundary of size at most $|N(S)\cap X| + |Y\setminus S| < |Y| \le 3k$ and $|X'| < |X|$, contradicting the minimality of $|X|$. This proves the claim.

Let $G' = G[X\cup Y]/M$. If $G'$ is $k$-connected, then the desired outcome of the lemma holds, contradicting that $G$ is a counterexample. So we may assume that $G'$ is not $k$-connected. More formally, that is, there exist $A',B'\subseteq V(G')$ such that $A'\setminus B', B'\setminus A' \ne \emptyset$, $A'\cup B' = V(G')$, $|A'\cap B'| \leq k-1$, and $\e(G'(A'\setminus B', B'\setminus A'))=0$. Let $A$ be the subset of $V(G)$ corresponding to $A'$ in $G'$, and similarly let $B$ be the subset of $V(G)$ corresponding to $B'$ in $G'$. 

Since $A'\setminus B'\ne \emptyset$ and $M$ saturates $Y$, it follows that $(A\cap X)\setminus B \ne\emptyset$. Similarly, $(B\cap X)\setminus A \ne \emptyset$.

Now $|A\cap B \cap Y| \leq |A'\cap B'| \leq k-1$. Hence $|A\cap Y| + |B\cap Y| \le |Y| + |A\cap B\cap Y| \le 4k-1$. Moreover, $|A\cap B\cap X|\leq k-1$. 

Yet if $|A\cap (Y\cup B)| \le 3k$, then $A\setminus (Y\cup B) = (A\cap X)\setminus B$ contradicts the minimality of $X$. So we may assume that $|A\cap (Y\cup B)| > 3k$. But then $|A\cap Y| > 2k$. Similarly, $|B\cap Y| > 2k$ as otherwise $B\setminus (Y\cup A) = (B\cap X)\setminus A$ contradicts the minimality of $X$. But now $|A\cap Y| + |B\cap Y| > 4k$, a contradiction.
\end{proof}

This yields the following structural corollary for $K_t$-minor-free graphs.

\begin{cor}\label{c:connect}
For every $\delta > 0$ and $1/2 \geq  \beta > 1/4$, there exists $C=C_{\ref{c:connect}}(\beta,\delta)>0$ such that for every $t \geq 3$ if a non-empty graph $G$ has no $K_t$ minor, then there exists a non-empty subset $X$ of $V(G)$ such that coboundary of $X$ has size at most $Ct(\log t)^{\beta}$ and $|X| \le t(\log t)^{3-5\beta + \delta}$.
\end{cor}
\begin{proof} We show that $C=6(C_{\ref{t:connect}}(\beta,\delta)+1)$  satisfies the corollary.
Let $d=Ct(\log t)^{\beta}$. If there exists a vertex $v\in V(G)$ with degree at most $d$ in $G$, then $X=\{v\}$ is as desired. So we may assume that $G$ has minimum degree at least $d$. Let $k = \lfloor d/6 \rfloor \geq C_{\ref{t:connect}}(\beta,\delta) t(\log t)^{\beta}$. By Lemma~\ref{l:contract}, there exists a non-empty $X\subseteq V(G)$ and a matching $M$ from the coboundary $Y$ of $X$ to $X$ that saturates $Y$ such that $|Y|\le 3k < d$ and $G'=G[X\cup Y]/M$ is $k$-connected. By Theorem~\ref{t:connect}, $\v(G') \le t (\log t)^{3-5\beta+\delta}$. Hence $|X| \le t (\log t)^{3-5\beta+\delta}$ as desired.
\end{proof}

We are now ready to prove Theorem~\ref{t:main}, which we restate for convenience. \Main*

\begin{proof}
Let $C=C_{\ref{c:connect}}(\beta,1)$. We show that for $t \gg C$, every graph $G$ with no $K_t$ minor is $2\lceil Ct(\log t)^{\beta} \rceil$-list-colorable, which implies the theorem.
Let $d = \lceil Ct(\log t)^{\beta} \rceil$. Suppose for a contradiction that there exists $G$ with no $K_t$ minor and a $2d$-list assignment $L$  such that $G$ is not $L$-colorable, and choose such a graph $G$ with $\v(G)$ minimum. By the choice of $C$, there exists a non-empty $X\subseteq V(G)$ such that coboundary $Y$ of $X$ has size at most $d$ and $|X| \le Ct(\log t)^{4-5\beta}$. By minimality, there exists an $L$-coloring $\phi$ of $G-X$. For each $v\in X$, let $L'(v) = L(v)\setminus \{\phi(w): w\in N(v)\setminus X\}$. Since $N(v)\setminus X\subseteq Y$ for each $v\in X$ by definition of coboundary, we have that $|N(v)\setminus X| \le |Y|\le d$. Hence for each $v\in X$, $|L'(v)|\ge |L(v)|- |Y| \ge d$. By Corollary~\ref{c:listsmall}, we have that $\chi_{\ell}(G[X]) \le C_{\ref{c:listsmall}}t \log^2(C(\log t)^{4-5\beta}) \le d$ for large enough $t$. Hence $G[X]$ has an $L'$-coloring $\phi'$. But now $\phi\cup \phi'$ is an $L$-coloring of $G$, a contradiction.
\end{proof}

\section{Further Improvements}\label{s:remarks}


The central obstacle in improving the bound on the chromatic number (and the list chromatic number) of  graphs with no $K_t$ minors using our methods is the absence of the analogue of \cref{t:connect} for graphs of connectivity $o(t (\log t)^{1/4})$. To determine the limits of this strategy it would be interesting to answer the following question.

\begin{que}
	For which $\beta >0 $, does there exist $C>0$ such that for every integer $t \geq 3$, every graph $G$ with $\kappa(G) = \Omega(t(\log t)^\beta)$ and no $K_t$ minor satisfies $\v(G) \leq t (\log t)^C$?\footnote{The last condition can be replaced by weaker, but less transparent inequality  $$\v(G) \leq t e^{o(\log t)}.$$}
\end{que} 

\subsubsection*{Acknowledgement.} We thank Zi-Xia Song for valuable comments.

 
\bibliographystyle{alpha}
\bibliography{../snorin}

\begin{thebibliography}{BKMM09}

\bibitem[Alo92]{Alon92}
Noga Alon.
\newblock Choice numbers of graphs: a probabilistic approach.
\newblock {\em Combin. Probab. Comput.}, 1(2):107--114, 1992.

\bibitem[BCE80]{BCE80}
B{\'e}la Bollob{\'a}s, Paul~A Catlin, and Paul Erd{\"o}s.
\newblock Hadwiger's conjecture is true for almost every graph.
\newblock {\em Eur. J. Comb.}, 1(3):195--199, 1980.

\bibitem[BJW11]{BJW11}
J\'{a}nos Bar\'{a}t, Gwena\"{e}l Joret, and David~R. Wood.
\newblock Disproof of the list {H}adwiger conjecture.
\newblock {\em Electron. J. Combin.}, 18(1):Paper 232, 7, 2011.

\bibitem[BKMM09]{BKMM09}
Thomas B\"{o}hme, Ken-ichi Kawarabayashi, John Maharry, and Bojan Mohar.
\newblock Linear connectivity forces large complete bipartite minors.
\newblock {\em J. Combin. Theory Ser. B}, 99(3):557--582, 2009.

\bibitem[DM82]{DucMey82}
Pierre Duchet and Henri Meyniel.
\newblock On hadwiger's number and the stability number.
\newblock In {\em North-Holland Mathematics Studies}, volume~62, pages 71--73.
  Elsevier, 1982.

\bibitem[Had43]{Had43}
Hugo Hadwiger.
\newblock \"{U}ber eine {K}lassifikation der {S}treckenkomplexe.
\newblock {\em Vierteljschr. Naturforsch. Ges. Z\"urich}, 88:133--142, 1943.

\bibitem[Kaw07]{Kaw07}
Ken-ichi Kawarabayashi.
\newblock On the connectivity of minimum and minimal counterexamples to
  {H}adwiger's {C}onjecture.
\newblock {\em J. Combin. Theory Ser. B}, 97(1):144--150, 2007.

\bibitem[KM07a]{KawMoh06}
K.~Kawarabayashi and Bojan Mohar.
\newblock Some recent progress and applications in graph minor theory.
\newblock {\em Graphs Combin.}, 23(1):1--46, 2007.

\bibitem[KM07b]{KawMoh07}
Ken-ichi Kawarabayashi and Bojan Mohar.
\newblock A relaxed {H}adwiger's conjecture for list colorings.
\newblock {\em J. Combin. Theory Ser. B}, 97(4):647--651, 2007.

\bibitem[Kos82]{Kostochka82}
A.~V. Kostochka.
\newblock The minimum {H}adwiger number for graphs with a given mean degree of
  vertices.
\newblock {\em Metody Diskret. Analiz.}, (38):37--58, 1982.

\bibitem[Kos84]{Kostochka84}
A.~V. Kostochka.
\newblock Lower bound of the {H}adwiger number of graphs by their average
  degree.
\newblock {\em Combinatorica}, 4(4):307--316, 1984.

\bibitem[NPS]{NPS19}
Sergey Norin, Luke Postle, and Zi-Xia Song.
\newblock Breaking the degeneracy barrier for coloring graphs with no {$K_t$}
  minor.
\newblock manuscript, 2019.

\bibitem[NS19]{NorSong19Odd}
Sergey Norin and Zi-Xia Song.
\newblock A new upper bound on the chromatic number of graphs with no odd
  {$K_t$} minor.
\newblock 2019.
\newblock arXiv:1912.07647.

\bibitem[RS98]{ReeSey98}
Bruce Reed and Paul Seymour.
\newblock Fractional colouring and {H}adwiger's conjecture.
\newblock {\em J. Combin. Theory Ser. B}, 74(2):147--152, 1998.

\bibitem[Sey16]{Sey16Survey}
Paul Seymour.
\newblock Hadwiger's conjecture.
\newblock In {\em Open problems in mathematics}, pages 417--437. Springer,
  2016.

\bibitem[Tho84]{Thomason84}
Andrew Thomason.
\newblock An extremal function for contractions of graphs.
\newblock {\em Math. Proc. Cambridge Philos. Soc.}, 95(2):261--265, 1984.

\bibitem[Tho01]{Thomason01}
Andrew Thomason.
\newblock The extremal function for complete minors.
\newblock {\em J. Combin. Theory Ser. B}, 81(2):318--338, 2001.

\bibitem[Voi93]{Voigt93}
Margit Voigt.
\newblock List colourings of planar graphs.
\newblock {\em Discrete Math.}, 120(1-3):215--219, 1993.

\end{thebibliography}

\end{document}